\documentclass[11pt]{amsart}
\usepackage{graphicx,multirow}
\usepackage[mathcal]{euscript}
\usepackage{float}
\usepackage{amsmath,amsthm,amssymb,amsfonts,amscd,epsfig,latexsym,graphicx,textcomp}

\restylefloat{figure}

\newtheorem{theorem}{Theorem}[section]
\newtheorem{lemma}[theorem]{Lemma}

\theoremstyle{definition}
\newtheorem{definition}[theorem]{Definition}

\theoremstyle{remark}

\numberwithin{equation}{section}

\newcommand{\BR}{\mathbb{R}}

\begin{document}

\title[Legendrian Cube Number]{An infinite family of Legendrian torus knots distinguished by cube number}

\author[B. McCarty]{Ben McCarty}

\address{Department of Mathematics, Louisiana State University \newline
\hspace*{.375in} Baton Rouge, LA 70817, USA} \email{\rm{benm@math.lsu.edu}}

\date{}

\begin{abstract}
For a knot $K$ the cube number is a knot invariant defined to be the smallest $n$ for which there is a cube diagram of size $n$ for $K$.  There is also a Legendrian version of this invariant called the \emph{Legendrian cube number}.  We will show that the Legendrian cube number distinguishes the Legendrian left hand torus knots with maximal Thurston-Bennequin number and maximal rotation number from the Legendrian left hand torus knots with maximal Thurston-Bennequin number and minimal rotation number.
\end{abstract}

\maketitle

\bigskip
\section{Introduction}
\bigskip

Cube diagrams are $3$-dimensional representations of knots or links (c.f. \cite{Adam}).  A cube diagram is a cubic lattice knot embedded in an $n \times n \times n$ cube in $\mathbb{R}^3$ where each projection to an axis plane ($x = 0$, $y = 0$, and $z = 0$) is a grid diagram.  The integer $n$ is the \emph{size} of the cube diagram and the \emph{cube number} of a knot, denoted $c(K)$, is the smallest $n$ for which there is a cube diagram for the knot of size $n$.  

In \cite{BaldMcCar} small examples of cube diagrams of knots were given up to $7$ crossings.  Some examples given were observed to be minimal but only those knots $K$ for which the cube number equaled the arc index, or $\alpha(K)$.  In general arc index is a lower bound for cube number.  Of interest are the cases where the inequality between arc index and cube number is strict.  In \cite{McCarty} this question was explored further.  It was shown that for $8$ of the first $12$ chiral knots in Rolfsen's knot table, cube number distinguishes betwen mirror images of knots.  

Let $K$ be a Legendrian knot.  Define the \emph{Legendrian cube number} (or cube number when the context is clear), $c_\ell(K)$, to be the minimum $n$ such that there is a cube diagram for $K$ of size $n$ that projects to a Legendrian front of $K$ (see details in Section \ref{section:legendrian}).  Perhaps surprisingly, the Legendrian cube number can distinguish between Legendrian knots of the same topogical type in some cases.  This fact is unexpected because there is no clear relationship as of yet between cube diagrams and Legendrian knots (cube diagrams do not naturally embed as Legendrian knots even when they have the same Legendrian knot projections).  In \cite{McCarty} it was proved that the Legendrian cube number distinguishes between two Legendrian $(5,2)$ torus knots.  In this paper we prove that this fact holds in general.  Specifically, we prove the following:

\bigskip
\noindent
{\bf Theorem 1} \emph{Let $p \geq 5$, $K_{min}$ be the left hand $(p,2)$-torus knot with maximal Thurston-Bennequin number and rotation number, $r(K_{min}) = 2 - p$ and $K_{max}$ the $(p,2)$-torus knot with maximal Thurston-Bennequin number and $r(K_{max}) = p - 2$.  Then the Legendrian cube number distinguishes between $K_{min}$ and $K_{max}$.}
\bigskip

\bigskip
\section{Definition of a cube diagram}

Let $n \in \mathbb{Z}^{+}$ and $\Gamma$ an $n \times n \times n$ cube, thought of as a $3$-dimensional Cartesian grid with integer-valued vertices.  A \textit{flat of $\Gamma$} is any cuboid (a right rectangular prism) with integer vertices in $\Gamma$ such that there are two orthogonal edges of length $n$ with the remaining orthogonal edge of length $1$.  A flat with an edge of length 1 that is parallel to the $x$-axis, $y$-axis, or $z$-axis is called an {\em $x$-flat}, {\em $y$-flat}, or {\em $z$-flat} respectively.  Note that the cube itself is canonically oriented by the standard orientation of $\BR^3$ (right hand orientation).

\begin{center}
\begin{figure}[h]
\centering
\includegraphics[scale=.3]{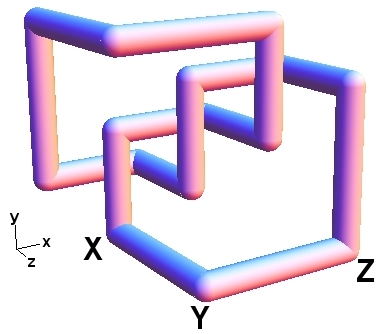}
\caption{Lefthand trefoil cube diagram.}
\label{fig:LHTrefoil}
\end{figure}
\end{center}

A \emph{marking} is a labeled half-integer point in $\Gamma$.  We mark unit cubes of $\Gamma$ with either an $X$, $Y$, or $Z$ such that the following {\em marking conditions} hold:
\begin{itemize}
    \item each flat has exactly one $X$, one $Y$, and one $Z$ marking;\\

    \item the markings in each flat form a right angle such that each segment is parallel to a coordinate axis;\\

    \item for each $x$-flat, $y$-flat, or $z$-flat, the marking that is the vertex of the right angle is an $X, Y,$ or $Z$ marking respectively.
\end{itemize}

We get an oriented link in $\Gamma$ by connecting pairs of markings with a line segment whenever two of their corresponding coordinates are the same.  Each line segment is oriented to go from an $X$ to a $Y$, from a $Y$ to a $Z$, or from a $Z$ to an $X$. The markings in each flat define two perpendicular segments of the link $L$ joined at a vertex, call the union of these segments a {\it cube bend}. If a cube bend is contained in an $x$-flat, we call it an {\it $x$-cube bend}. Similarly, define {\it $y$-cube bends} and {\it $z$-cube bends}.

\begin{center}
\begin{figure}[h]
\centering
\includegraphics[scale=.2]{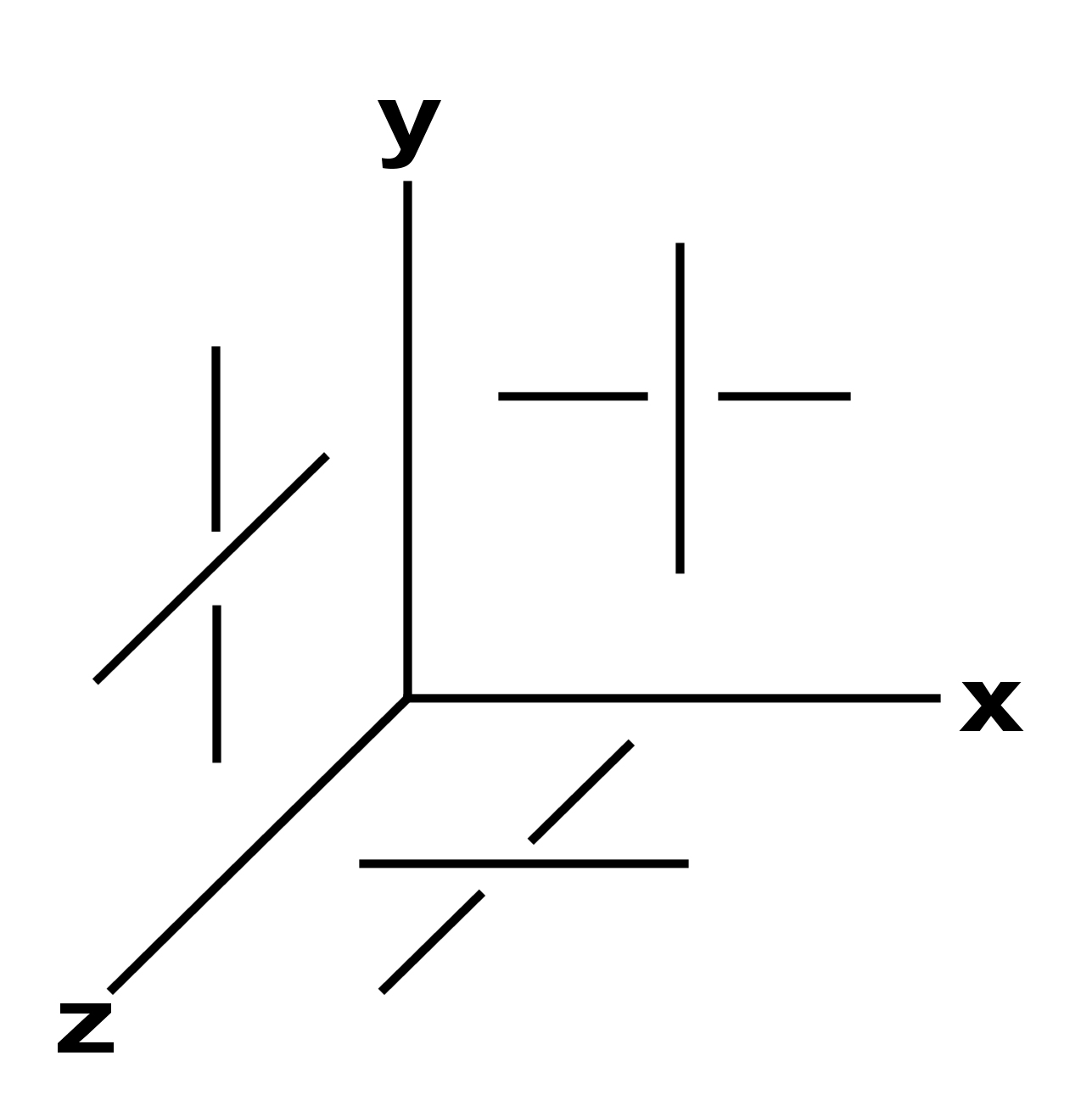}
\caption{Crossing conditions of the knot at every intersection in each projection.}
\label{fig:crossings}
\end{figure}
\end{center}

Arrange the markings in $\Gamma$ so that at every intersection point of the $(x,y)$-projection (i.e., $\pi_z : \mathbb{R}^3 \rightarrow \mathbb{R}^3$ given by $\pi_z(x,y,z)=(x,y)$), the segment parallel to the $x$-axis has smaller $z$-coordinate than the segment parallel to the $y$-axis.  Similarly, arrange so that in the $(y,z)$-projection, $z$-parallel segments cross over the $y$-parallel segments, and in the $(z,x)$-projection, the $x$-parallel segments cross over the $z$-parallel segments (see Figure \ref{fig:crossings}).

A set of markings in $\Gamma$ satisfying the marking conditions and crossing conditions is called a \emph{cube diagram} for the knot or link.  Note that the definition presented here differs from the one presented in \cite{Bald} by a shift of the markings:  change $Z$ to $Y$, $Y$ to $X$, and $X$ to $Z$.

\bigskip
\section{Liftability of grid diagrams}

Because cube diagrams project to grid diagrams, it is natural to think of a cube diagram as a lift of a grid diagram corresponding to, say, the $(x,y)$-projection of the cube.  However, such lifts do not always exist (c.f. \cite{Adam} and \cite{BaldMcCar}).  

Before proceeding, we need to establish some terminology and facts about grid diagrams (for more details see \cite{Adam}).  Grid diagrams will be oriented so that vertical edges are directed from $X$ to $O$.  A \emph{bend} in a grid diagram, $G$, is a pair of segments that meet at a common $X$ or $O$ marking.  We will refer to the former pair of segments as an $X$-bend and the latter as an $O$-bend.  There are two ways to decompose any link component in $G$ into a set of non-overlapping bends, corresponding to a choice of $X$-bends or $O$-bends.  In particular, for a knot there are only two ways to decompose $G$ into non-overlapping bends, and such a choice will be called a \emph{bend decomposition}. 

Consider a grid diagram, $G$, together with a choice of a bend decomposition.  If possible we wish to lift $G$ to a cube diagram where $G$ is the $(x,y)$-projection of the cube diagram and the bend decomposition of $G$ determines the $z$-cube bends of the cube diagram.  While $G$ carries with it an orientation on the knot, so does the $(x,y)$-projection of the cube diagram.  In order that these orientations agree, the $X$-bend decompositon of $G$ has to be chosen--$O$-bends cannot be lifted to $z$-cube bends.  Furthermore, because of the symmetry between all three projections in a cube diagram, it is enough to work only with the $(x,y)$-projection and lift $X$-bends to $z$-cube bends.  

The crossings in a grid diagram sometimes generate a \emph{partial order} on the $X$-bends.  Let $b_1$ and $b_2$ be two $X$-bends.  If $b_1$ crosses over $b_2$ in $G$ we say that $b_1 > b_2$.  Thus in any lift of $G$, the $z$-cube bend corresponding to $b_1$ must have $z$-coordinate greater than that of the $z$-cube bend corresponding to $b_2$.  

Of course, not every grid diagram has a partial order on the $X$-bends.  A grid diagram for which there is no partial order on the $X$-bends may not even lift to a lattice knot that has well-defined knot projections in the other planes (Figure 5 of \cite{Adam}).  However, if there is a partial ordering on the $X$-bends of the grid diagram, it will lift to a lattice knot in which all projections are well-defined knot projections (c.f. \cite{Adam}).  Nevertheless, even a partial order doesn't guarantee liftability to a cube diagram as the $(y,z)$- and $(z,x)$-projections may not be grid diagrams in such a lift (c.f \cite{Adam} and \cite{BaldMcCar}).  Below, we will introduce some grid configurations that fail to lift, not because of a lack of partial ordering but due to crossings in the $(y,z)$- or $(z,x)$-projections that do not satisfy the crossing conditions for a cube diagram.  In Figures \ref{fig:type1} and \ref{fig:type2}, the shaded regions are determined by the corresponding $X$-bend and extend from the $X$-bend to the boundary of the grid diagram as indicated.  Furthermore, a dotted edge represents a sequence of edges in the grid that remains in the shaded region.  This condition guarantees that at least one $z$-parallel edge will introduce a crossing in either the $(y,z)$ or $(z,x)$-projection which does not follow the crossing condition (c.f. \cite{McCarty}).

\begin{theorem}
\label{thm:type1}
The Type 1 and 2 configurations shown in Figures \ref{fig:type1} and \ref{fig:type2} do not appear in the projection of a cube diagram.
\end{theorem}

\begin{center}
\begin{figure}[h]
\centering
\includegraphics[scale=.3]{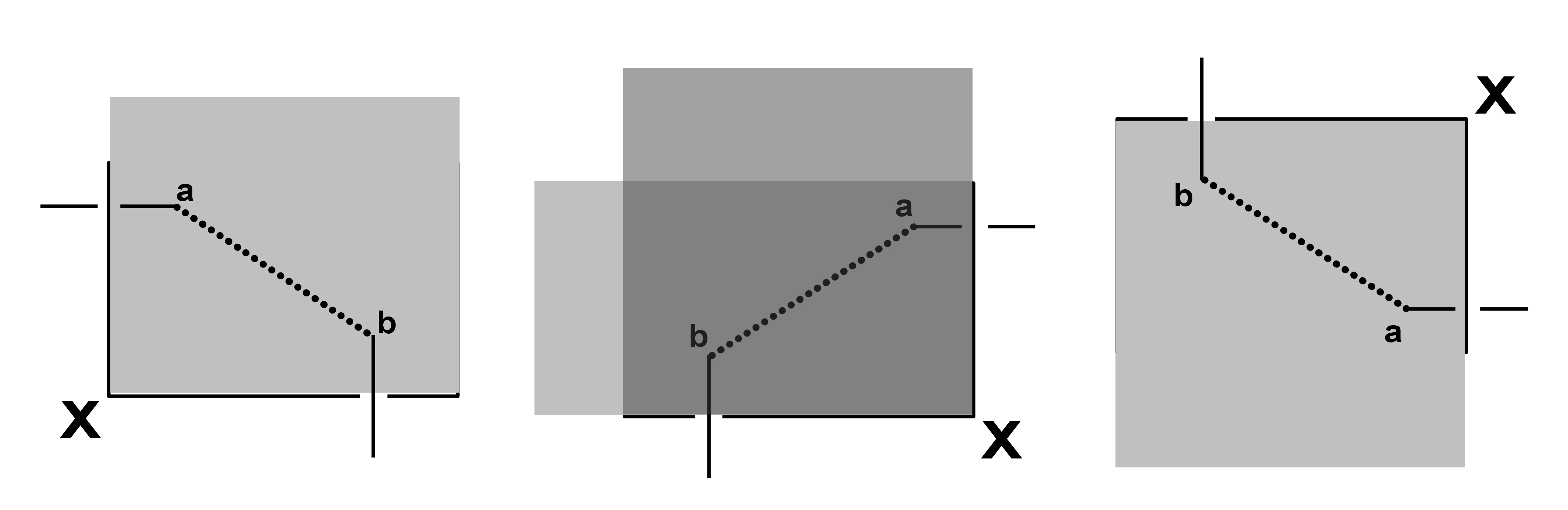}
\caption{Type 1 configurations.}
\label{fig:type1}
\end{figure}
\end{center}

\begin{center}
\begin{figure}[h]
\centering
\includegraphics[scale=.3]{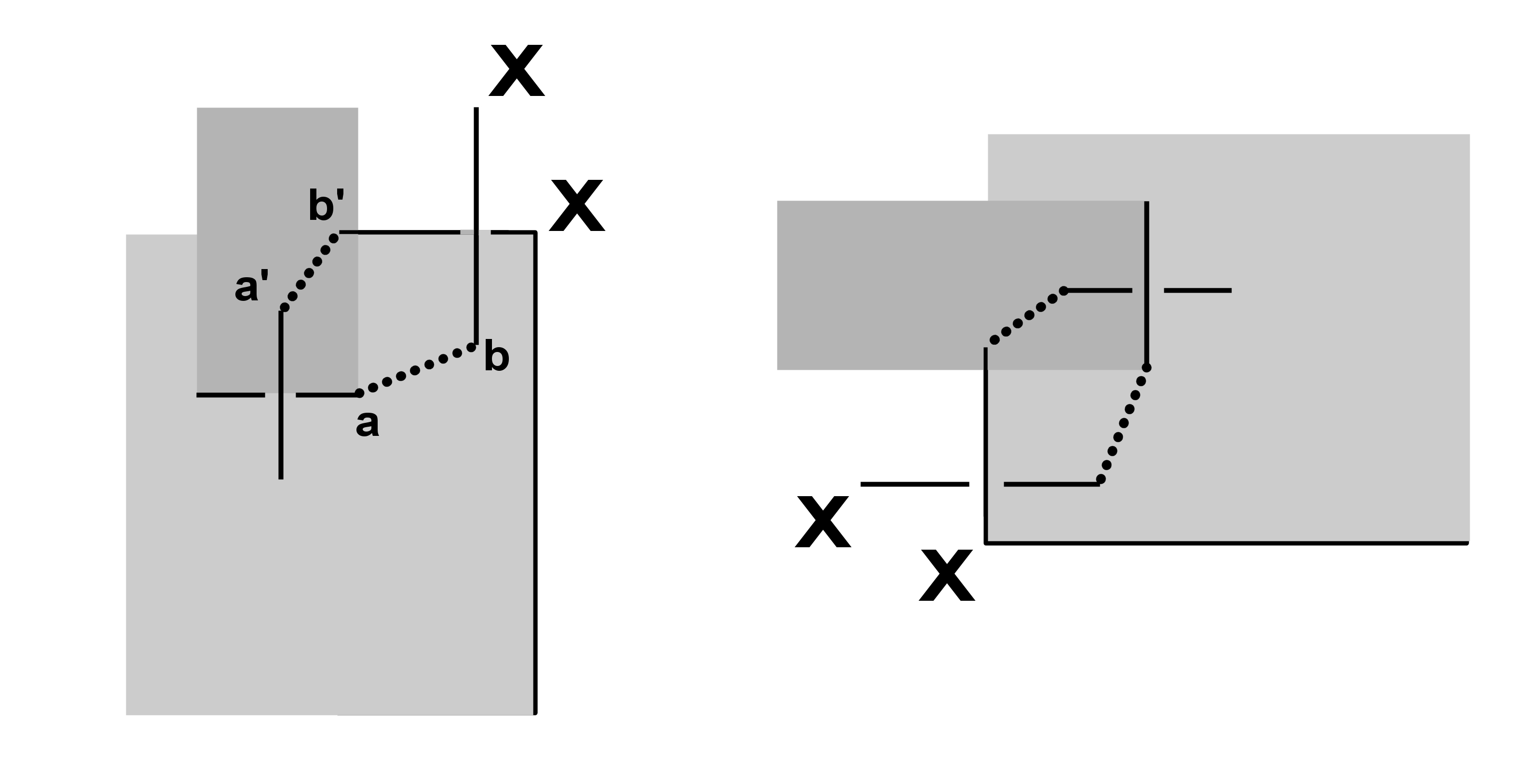}
\caption{Type 2 configurations.}
\label{fig:type2}
\end{figure}
\end{center}

\bigskip

\section{Cube number and Legendrian type}
\label{section:legendrian}

Any grid diagram represents the front projection of a Legendrian knot by following this procedure.  First smooth the northeast and southwest corners.  Then convert northwest and southeast corners to cusps and rotate the grid diagram $45$ degrees counterclockwise.  Alternatively, to obtain a Legendrian front projection for the mirror image of the knot represented by the given grid diagram, reverse all crossings, rotate the grid $45$ degrees clockwise, convert northeast and southwest corners to cusps and smooth the remaining corners.  While there is no similar construction to convert a cube diagram into a Legendrian knot, each of the projections of a cube diagram is a grid diagram, and hence, represents the Legendrian front projection of some knot.  Therefore one can define the Legendrian cube number, $c_\ell(K)$, to be the smallest $n$ such that there is a cube diagram for the knot $K$ of size $n$ where the $(x,y)$-projection of the cube diagram is a grid diagram representing the Legendrian knot $K$.  

It is not immediately obvious that the Legendrian cube number is defined for all Legendrian knots.  The construction given in \cite{Adam} shows how to lift any grid diagram (up to stabilizations) to a cubic lattice knot satisfying the marking conditions for a cube diagram.  The same construction may be done using only stabilizations of the grid that preserve the Legendrian type of the front projection represented by the grid.  Given a grid diagram $G$ representing a Legendrian knot $K$ one may perform Legendrian stabilizations and lift the diagram to a cubic lattice knot satisfying the marking conditions of a cube diagram, and the crossing conditions of the $(x,y)$-projection.  All that remains is to show that the crossing conditions of the other two projections may be corrected by a \emph{twisted crossing}.  Figures \ref{fig:twistedcrossing} and \ref{fig:TwistedxyProjection} show how to insert such a correction in the $(y,z)$-projection.  The construction for the $(z,x)$-projection is similar.  Note that Figure \ref{fig:twistedcrossing} is almost the same as the twisted crossing given in \cite{Adam} but has been modified slightly so that the stabilizations in the $(x,y)$-projection are Legendrian.  

\begin{center}
\begin{figure}[h]
\centering
\includegraphics[scale=.3]{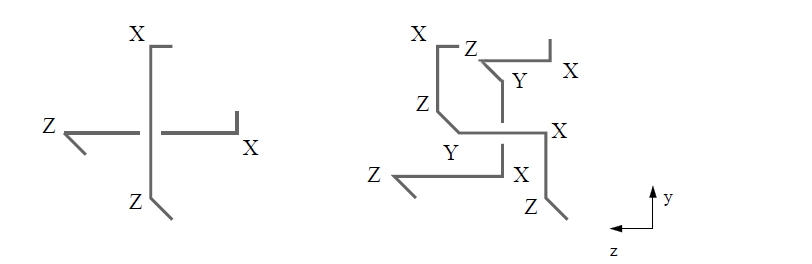}
\caption{Inserting a twisted crossing.}
\label{fig:twistedcrossing}
\end{figure}
\end{center}

\begin{center}
\begin{figure}[h]
\centering
\includegraphics[scale=.5]{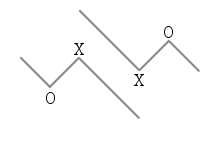}
\caption{The Legendrian stabilizations in the $(x,y)$-projection after the insertion of a twisted crossing.}
\label{fig:TwistedxyProjection}
\end{figure}
\end{center}

According to \cite{etnyrehonda} the \emph{rotation number} may be computed from the Legendrian front projection as follows:
$$r(K) = \frac{1}{2}(D_c - U_c)$$
where $D_c$ is the number of downward oriented cusps and $U_c$ is the number of upward oriented cusps in the Legendrian front projection.  Also, the \emph{Thurston-Bennequin number} of a Legendrian knot may be computed as follows:
$$tb(K) = \omega(K) - \frac{1}{2} ( D_c + U_c)$$
where $\omega(K)$ is the writhe of the front projection of $K$.  Furthermore, according to \cite{etnyrehonda} and \cite{Ng}, any minimal grid diagram for a left hand Legendrian torus knot, $T_{p,2}$, must realize the maximal Thurston Bennequin number and hence have rotation number satisfying:
$$r(K) \in \left\{ \pm (p - 2 - 4t): t \in \mathbb{Z}, 0 \leq t < \frac{p - 2}{2} \right\}.$$


\begin{center}
\begin{figure}[h]
\centering
\includegraphics[scale=.4]{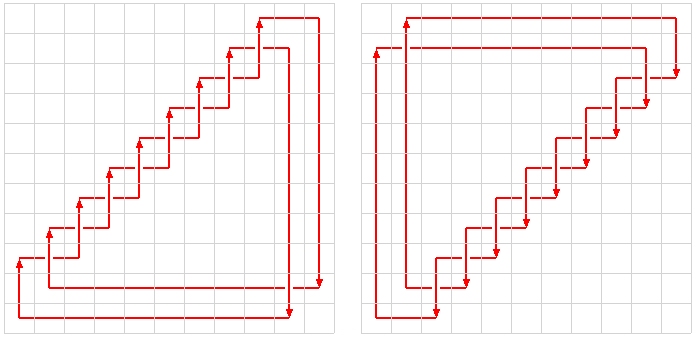}
\caption{Diagrams for the $(9,2)$ torus knot with $r = -7$ and $r = 7$ respectively.}
\label{fig:lhp2}
\end{figure}
\end{center}

\bigskip
\noindent
{\bf Theorem 1} \emph{Let $p \geq 5$, $K_{min}$ be the left hand $(p,2)$-torus knot with maximal Thurston-Bennequin number and rotation number, $r(K_{min}) = 2 - p$ and $K_{max}$ the $(p,2)$-torus knot with maximal Thurston-Bennequin number and $r(K_{max}) = p - 2$.  Then the Legendrian cube number distinguishes between $K_{min}$ and $K_{max}$.}
\bigskip

The proof of Theorem 1 will begin with a series of lemmas.  Given a front projection of a Legendrian knot with maximal Thurston-Bennequin number, we use Legendrian invariants to compute the number of maxima, minima and the number of downward and upward oriented cusps (Lemma \ref{bendtypes}).  We then get upper and lower bounds on what the writhe of the diagram can be (Lemma \ref{writhebound}).  Then, we show that $c_\ell(K_{max}) = \alpha(K) = p + 2$ for all $p$ (Lemma \ref{Kmax}).  Finally, we interpret what such a front would look like as a minimal grid diagram, and show that such a grid for $K_{min}$ will necessarily contain Type 1 configurations.  

Before proceeding we will define a partial order on the lattice points of a grid which will prove useful when thinking of grid diagrams as Legendrian front projections.  

\begin{definition} 
Given two points $P_1 = (x_1,y_1)$ and $P_2 = (x_2,y_2)$, we say that $P_1 \preceq P_2$ if and only if $x_1 \leq x_2$ and $y_1 \leq y_2$.  In this case we say that $P_1$ is \emph{below} $P_2$, or alternatively that $P_2$ is \emph{above} $P_1$ (see Figure \ref{fig:below}).
\end{definition}

Note that in Figure \ref{fig:below} $P_3$ is not comparable to to $P_1$.  Points that are comparable using this partial order may be connected by an arc that consists only of upward oriented cusps (thought of as coming from a grid diagram rotated to a Legendrian front).  Note that for a pair of points that are not comparable any path in the grid connecting them, will introduce a new local extremum.  

\begin{center}
\begin{figure}[h]
\centering
\includegraphics[scale=.3]{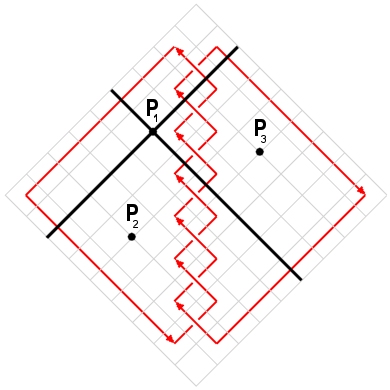}
\caption{$P_2$ is below $P_1$.}
\label{fig:below}
\end{figure}
\end{center}

Let $G$ be a minimal grid diagram for $K_{min}$.  Denote the number of northeast $X$-bends in $G$ by $X_{NE}$.  Similarly define $X_{SE}$, $X_{NW}$, $X_{SW}$, $O_{NE}$, $O_{SE}$, $O_{NW}$, and $O_{SW}$.  When converting $G$ to a left hand Legendrian front projection (i.e. a left hand torus knot) the number of downward oriented cusps will be $D_L = X_{NW} + O_{SE}$, and the number of upward oriented cusps will be $U_L = O_{NW} + X_{SE}$.  When converting $G$ to a right hand Legendrian front projection (i.e. a right hand torus knot) the number of downward oriented cusps will be $D_R = X_{NE} + O_{SW}$ and the number of upward oriented cusps will be $U_R = O_{NE} + X_{SW}$.

\begin{lemma}
\label{bendtypes}
For a grid diagram, $G$, representing a left hand $(p,2)$ torus knot, the number of bends is as follows:  $D_L = 2 + \omega + p$, $U_L = \omega + 3p - 2$, and $D_R = U_R = 2 - p - \omega$.  Furthermore, the number of maxima and minima in a Legendrian front corresponding to $G$ must be equal.
\end{lemma}

\begin{proof}
According to \cite{etnyrehonda} any Legendrian front projection for a right hand torus knot with maximal Thurston-Bennequin number has rotation number equal to $0$.  Hence, the number of downward oriented cusps equals the number of upward oriented cusps in the right hand Legendrian front obtained from $G$.  That is, $D_R = U_R$.  Furthermore, according to \cite{etnyrehonda} the maximal Thurston-Bennequin number of the left hand Legendrian front corresponding to $G$ is $-2p$.  Hence, for a minimal grid diagram (which must have maximal Thurston-Bennequin number according to \cite{Ng}) we have the following equation:
$$-2p = \omega(G) - \frac{1}{2}(D_L + U_L).$$
Also, since the minimal rotation number realizable in a minimal grid diagram is $2-p$ and the minimum rotation number must equal $\frac{1}{2}(D_L - U_L)$ we have the following:
$$D_L - U_L = 4-2p.$$
Solving for $D_L$ and $U_L$ we obtain:
$$D_L = 2 + \omega(G) + p$$
$$U_L = \omega(G) + 3p -2$$
where $\omega(G)$ is the writhe of the diagram.  Also, the total number of bends in a minimal grid $G$ of any type is $2(p+2) = D_R + U_R + D_L + U_L$.  Since $D_R = U_R$ we find:
$$D_R = U_R = 2 - p - \omega(G)$$

\begin{center}
\begin{figure}[h]
\centering
\includegraphics[scale=.3]{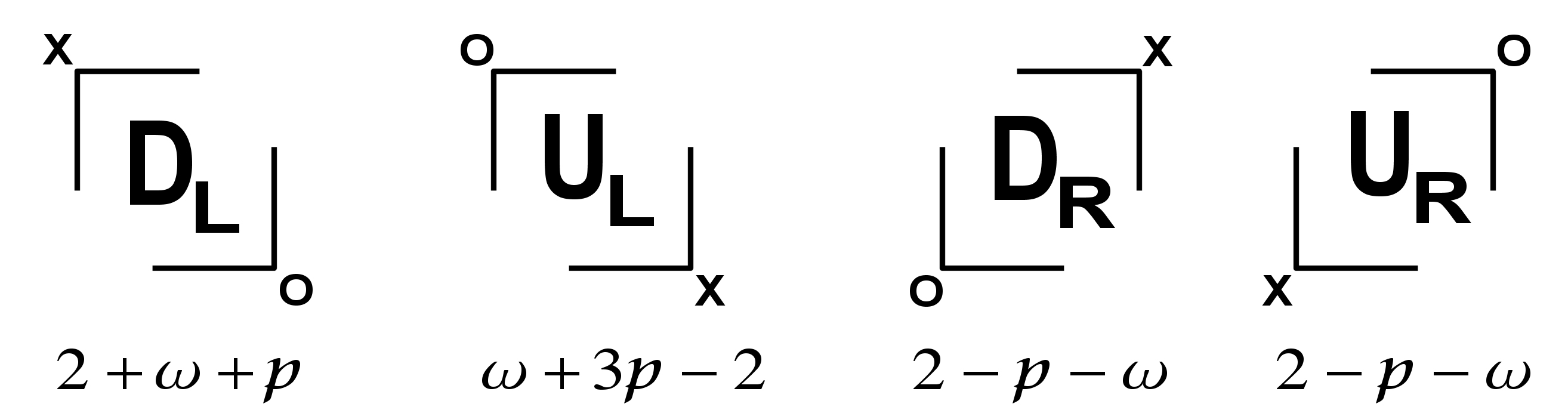}
\caption{The types of bends in $G$ with the number of each type that may occur where $\omega = \omega(G)$.}
\label{fig:bendtypes}
\end{figure}
\end{center}

For the last statement, since the Euler characteristic of $S^1$ is $0$ the number of index 1 critical points (relative maxima) and the number of index 0 critical points (relative minima) in a Legendrian front must be equal. 

\end{proof}

\begin{lemma}
\label{writhebound}
Given a grid diagram for $K_{min}$ we have the following bound on the writhe:  $-p - 2 \leq \omega(G) < 2 - p$.
\end{lemma}

\begin{proof}
Since any knot diagram must contain relative maxima and minima, $D_R > 0$ and hence by Lemma \ref{bendtypes}, $\omega(G) < 2 - p$.  Also, since $D_L \geq 0$, Lemma \ref{bendtypes} implies that $\omega(G) \geq -p - 2$.
\end{proof}

\begin{lemma}
\label{Kmax}
$c_\ell(K_{max}) = \alpha(T_{p,2}) = p + 2$.
\end{lemma}

\begin{proof}
Extend the construction shown in Figure \ref{fig:KMax} in the obvious way.
\end{proof}

\begin{center}
\begin{figure}[h]
\centering
\includegraphics[scale=.3]{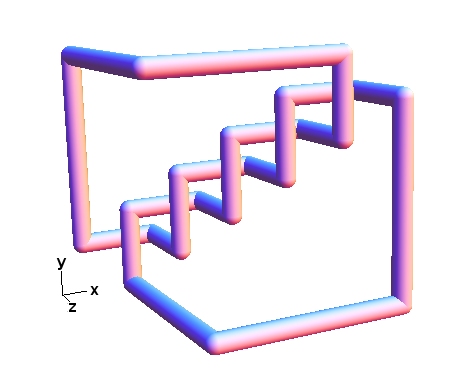}
\caption{A cube diagram for $K_{max}$ when $p = 5$.}
\label{fig:KMax}
\end{figure}
\end{center}

\begin{proof}[Proof of Theorem 1]

The remainder of the proof breaks down into four cases based on the value of $\omega(G)$ (c.f. Lemma \ref{writhebound}).

\bigskip

\noindent {\bf Case 1:  $\omega(G) = 1 - p$.}  

By Lemma \ref{bendtypes}, $D_L = 3$, $U_L = 2p - 1$ and $D_R = U_R = 1$.  After converting $G$ to a Legendrian front projection we obtain a knot diagram with exactly one maximum and one minimum.  Since $T_{p,2}$ is a two-bridge knot (c.f. \cite{Schultens}), any diagram must contain at least 2 maxima and minima, thus we obtain a contradiction.

\bigskip

\noindent {\bf Case 2:  $\omega(G) = -p$.}  

By Lemma \ref{bendtypes}, $D_L = 2$, $U_L = 2p - 2$, and $D_R = U_R = 2$.  A Legendrian front for $K_{min}$ has exactly $2$ relative maxima and $2$ relative minima.  There are three cases to consider: 

\medskip

\begin{enumerate}
  \item both relative maxima (and both relative minima) are marked by an $X$,
  \item both relative maxima (and both relative minima) are marked by an $O$,
  \item one relative maximum (respectively minimum) is marked by an $O$ and one relative maximum (respectively minimum) is marked by an $X$
\end{enumerate}

\begin{center}
\begin{figure}[h]
\centering
\includegraphics[scale=.7]{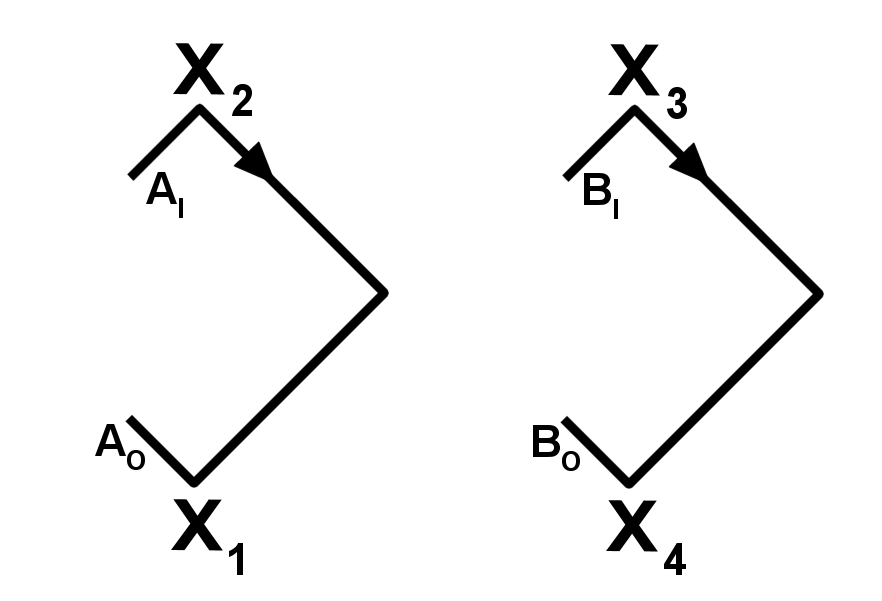}
\caption{There are two possible ways to connect the upward oriented arcs.}
\label{fig:2capsconnections}
\end{figure}
\end{center}

Note that for Case 1 (and by symmetry Case 2) the labels on the maxima and minima must all be the same, lest we have too many bends of type $D_R$ or $U_R$.  Therefore, each maximum must connect along a downward oriented arc to a minimum via a single downward oriented cusp (c.f. Figure \ref{fig:2capsconnections}).  There are two possibilities for how to connect the upward oriented arcs. Denote a connection between two endpoints with a colon.

\begin{enumerate}
\item $A_O : A_I$ and $B_O : B_I$.
\item $A_O : B_I$ and $B_O : A_I$.
\end{enumerate}

Since the first possibility creates two components, we do not consider it.  For the second case, since there are only $2$ downward oriented cusps the downward oriented arcs may either cross once, or not at all depending on how the two configurations shown in Figure \ref{fig:2capsconnections} are situated.  Note that since $A_O$ connects to $B_I$, $X_1 \preceq X_3$ and since $B_O$ connects to $A_I$, $X_4 \preceq X_2$ (see Figure \ref{fig:2capscomparison}).  Either $X_3$ will be comparable to $X_2$ or not.  If $X_3$ is not comparable to $X_2$ then we have one of the four cases shown in Figure \ref{fig:2capscomparison2} corresponding to the position of $X_1$ relative to $X_4$.  We dispense with these four cases by observing that in each case, there are not enough upward oriented bends to produce $K_{min}$.  We may then assume, without loss of generality, that $X_2 \preceq X_3$.  Then, we may also assume that $X_4 \preceq X_1$, following the same line of reasoning that we used to show that $X_2$ and $X_3$ must be comparable.  Since $X_2 \preceq X_3$ and $X_4 \preceq X_1$ the downward oriented arcs must cross once as shown in Figure \ref{fig:2capsXX2}.

\begin{center}
\begin{figure}[h]
\centering
\includegraphics[scale=.7]{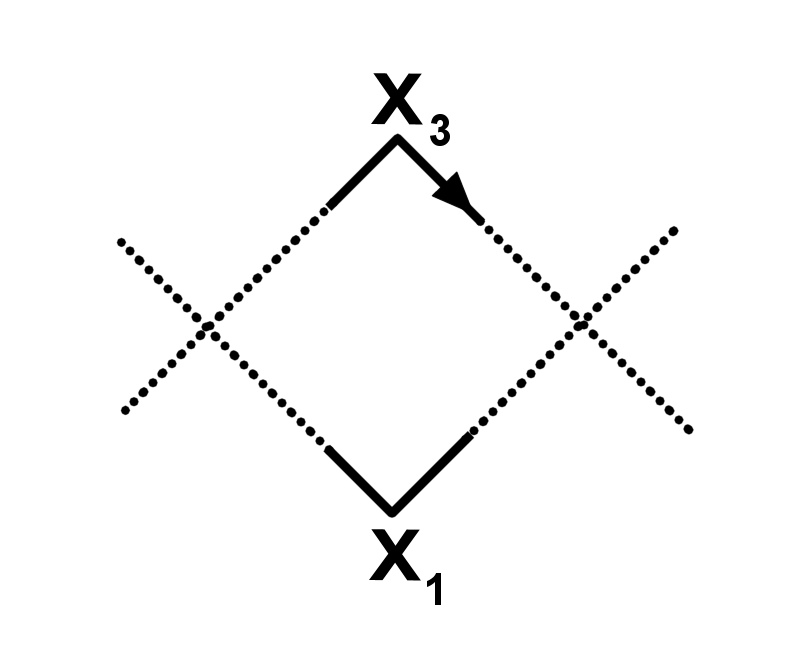}
\caption{$X_1 \preceq X_3$.}
\label{fig:2capscomparison}
\end{figure}
\end{center}

\begin{center}
\begin{figure}[h]
\centering
\includegraphics[scale=.7]{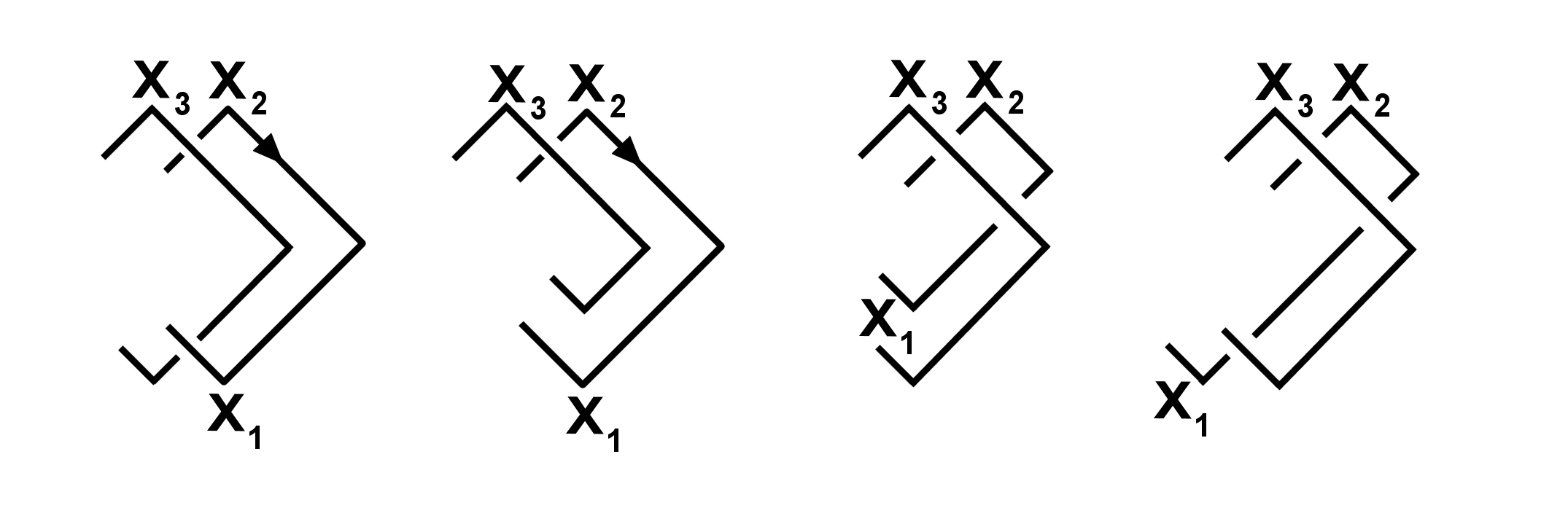}
\caption{Two cases where $X_3$ is not comparable to $X_2$.}
\label{fig:2capscomparison2}
\end{figure}
\end{center}

Thus, the upward oriented twisting arcs must complete at least $p - 1$ half twists in order to construct $K_{min}$, requiring all $2p - 2$ upward oriented cusps.  Since $p \geq 5$ there must be at least four half-twists which will necessarily create Type 1 configurations.  See Figure \ref{fig:2capsXX2} for an example of a Legendrian front for $K_{min}$.

\begin{center}
\begin{figure}[h]
\centering
\includegraphics[scale=.7]{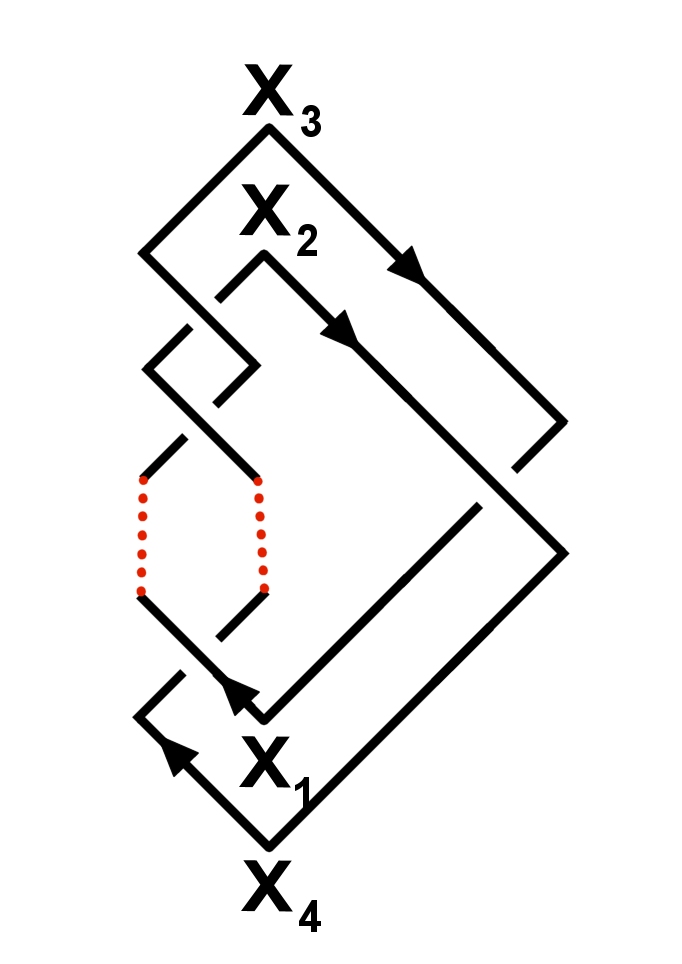}
\caption{One possibility for diagrams with $2$ relative maxima, both labeled with $X$.}
\label{fig:2capsXX2}
\end{figure}
\end{center}


For the third case, when one maximum is labeled with an $X$ and one is labeled with an $O$, we again consider whether the marked points are comparable or not.  If the maxima are not comparable, and the minima are not comparable, then the diagram will match either Figure \ref{fig:2capsXOconnections} or \ref{fig:2capsXOconnections2}.  For the diagram shown in Figure \ref{fig:2capsXOconnections} the downward oriented arcs must be positioned relative to each other such that the dotted lines cross below $X_2$ and $O_2$ lest the upward oriented arcs require the addition of a maximum to connect with $A_I$ and $B_I$ (see Figure \ref{fig:2capsDownFullTwist}).  In such a case, there are not enough upward oriented bends to create $K_{min}$.  

\begin{center}
\begin{figure}[h]
\centering
\includegraphics[scale=.7]{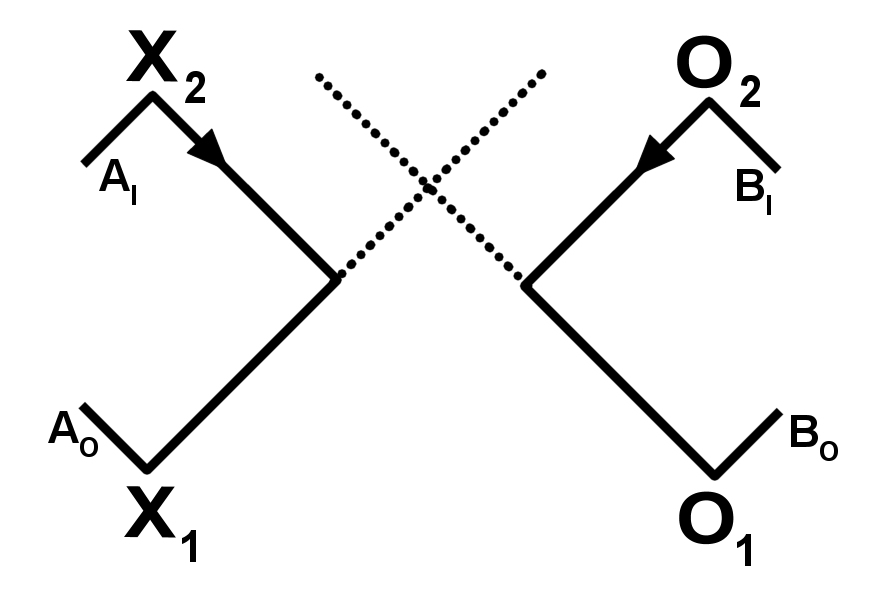}
\caption{Crossings of upward arcs must occur in the region above the dotted lines.}
\label{fig:2capsXOconnections}
\end{figure}
\end{center}

For the diagram shown in Figure \ref{fig:2capsXOconnections2}, the maxima must be above the crossing of the dotted lines, in order for the upward oriented arcs to connect up with $A_I$ and $B_I$.  In this case, the twisting of the upward arcs requires all $2p-2$ bends available (see Figure \ref{fig:2capsXO3}) and since $p \geq 5$ introduces Type 1 configurations.

\begin{center}
\begin{figure}[h]
\centering
\includegraphics[scale=.7]{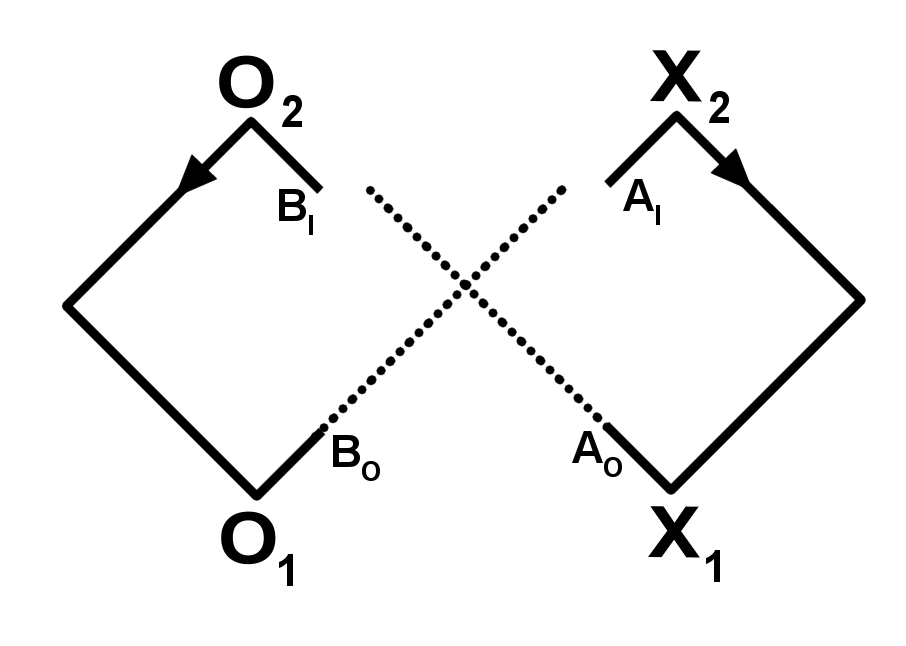}
\caption{Crossings of upward arcs must occur in the region above the dotted lines.}
\label{fig:2capsXOconnections2}
\end{figure}
\end{center}

\begin{center}
\begin{figure}[h]
\centering
\includegraphics[scale=.7]{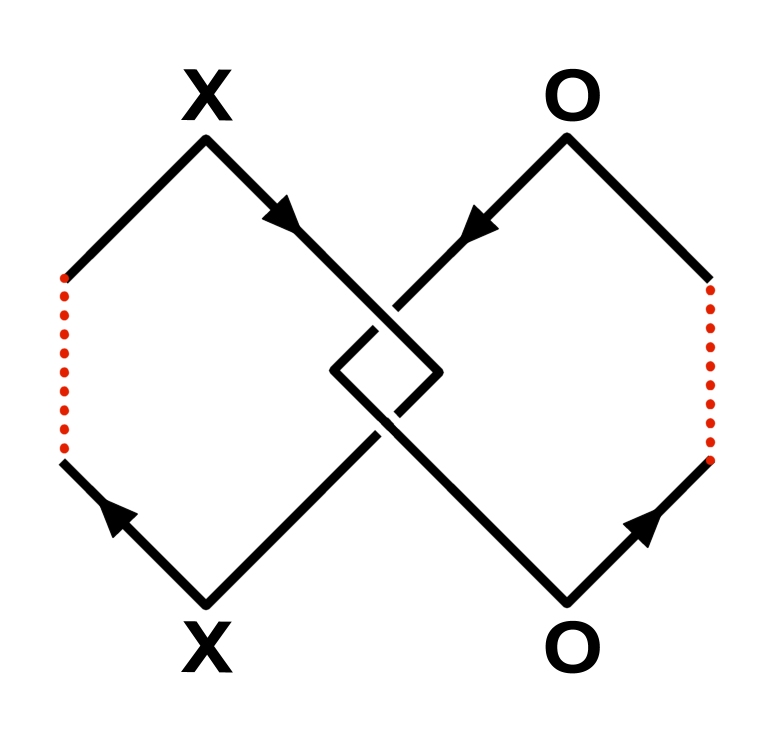}
\caption{Downward oriented strands completing one full twist.}
\label{fig:2capsDownFullTwist}
\end{figure}
\end{center}

If both downward oriented cusps lie on the same arc (as in Figure \ref{fig:2capsXO2}), then the twisting of the upward oriented arcs must occur above the crossing of the downward oriented arcs.  In this case, at least one of the arcs will require the addition a maximum to connect with a relative maximum.

\begin{center}
\begin{figure}[h]
\centering
\includegraphics[scale=.7]{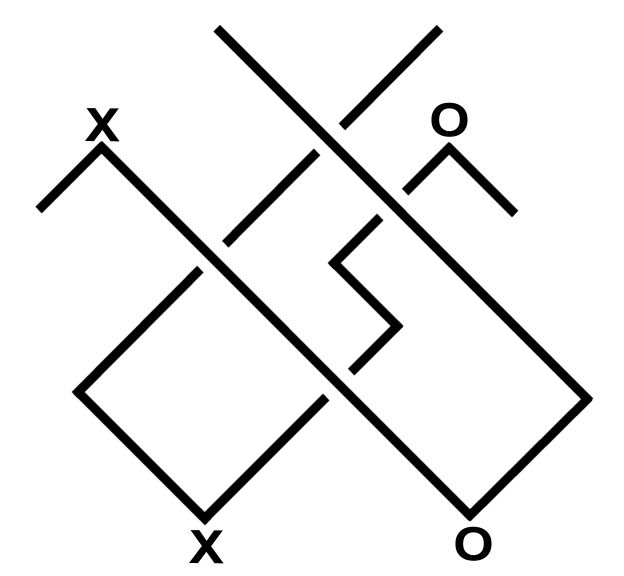}
\caption{One crossing on downward arcs.}
\label{fig:2capsXO2}
\end{figure}
\end{center}

\begin{center}
\begin{figure}[h]
\centering
\includegraphics[scale=.7]{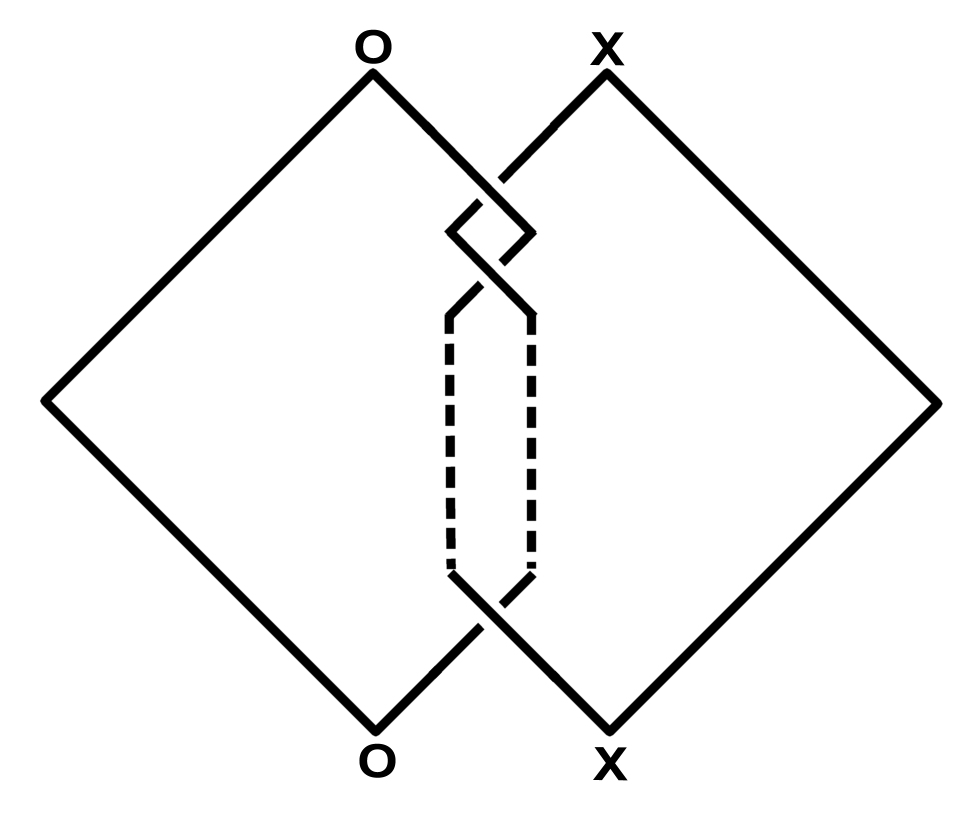}
\caption{Downward arcs with no crossing.}
\label{fig:2capsXO3}
\end{figure}
\end{center}

\bigskip
\noindent {\bf Case 3:  $\omega(G) = -p - 1$.}

In this case $D_L = 1$, $U_L = 2p - 3$, and $U_R = D_R = 3$.  Because the Legendrian front contains a single downward oriented cusp, two of the relative maxima must connect to two relative minima by a single edge each, while the third relative maximum connects to the third relative minimum via a single downward oriented cusp (both cases shown in Figure \ref{fig:3capsconnections}).  For the two pair of extrema connected by a single edge, if the two maxima are labeled with an $X$, then the two minima must be labeled with an $O$, and hence there would be $4$ bends of type $D_R$.  Thus, for these extrema there is one $X$ maximum and one $O$ maximum as shown in Figure \ref{fig:3capsconnections}.  For the upward oriented arcs there are three possibilities for how to connect the labelled endpoints in Figure \ref{fig:3capsconnections}.  Denoting a connection between two endpoints with a colon, the possibilities are:

\begin{enumerate}
 \item $B_O : A_I$, $C_O : B_I$, $A_O : C_I$,
 \item $C_O : A_I$, $A_O : B_I$, $B_O : C_I$,
 \item $B_O : A_I$, $A_O : B_I$, $C_O : C_I$.
\end{enumerate}

The first two possibilities lead to a single component, while the third produces more than one component.  The following is for the diagrams shown on the left in Figures \ref{fig:3capsconnections} and \ref{fig:3capsconnections2}.  The argument for the diagrams shown on the right in Figures \ref{fig:3capsconnections} and \ref{fig:3capsconnections2} is similar.

\begin{center}
\begin{figure}[h]
\centering
\includegraphics[scale=.7]{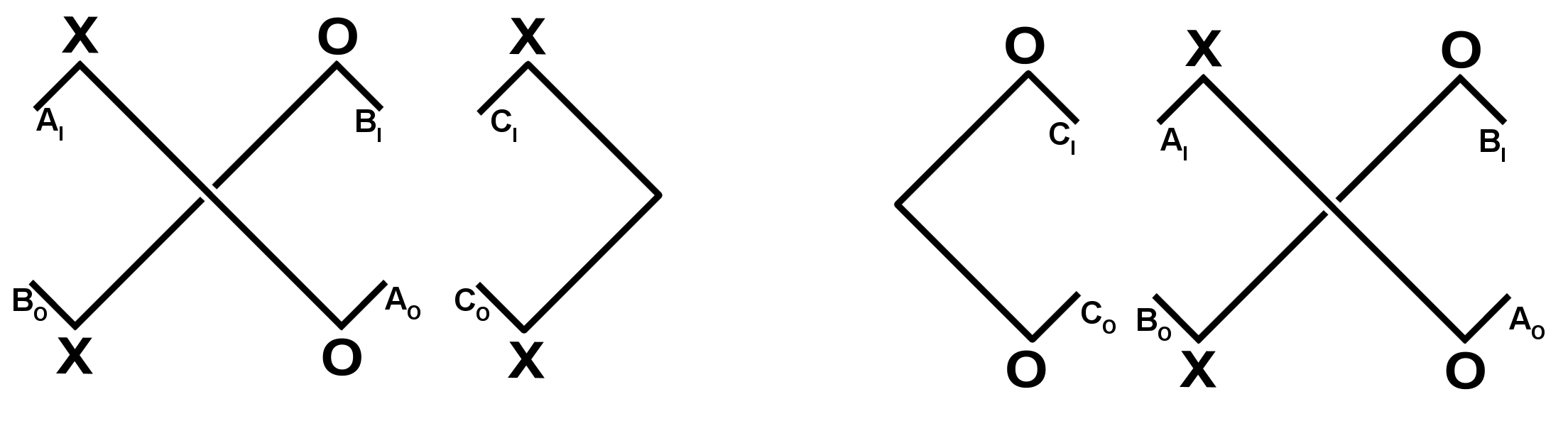}
\caption{$\omega(G) = -p - 1$.}
\label{fig:3capsconnections}
\end{figure}
\end{center}

For case $1$ refer to the arc connecting $A_O$ to $C_I$ by $\alpha$, the arc connecting $C_O$ to $B_I$ by $\beta$ and the arc connecting $B_O$ to $A_I$ by $\gamma$.  To form $K_{min}$ two of the upward oriented arcs must twist.  The $\beta$ and $\gamma$ arcs cannot twist since the entire $\beta$ arc must lie below the $O$ maximum and in order for the $\gamma$ arc to enter the region below the $O$ maximum, it would have to contain an additional relative maximum.  For similar reasons the $\alpha$ and $\gamma$ arcs cannot twist either.  Therefore any twisting that occurs must occur between the $\alpha$ and $\beta$ arcs.  The twisting of $\alpha$ and $\beta$ also means that $B_I$ must lie above $A_O$ and $A_I$ must lie above $B_O$, meaning that the downward oriented arcs connected to these ends must cross as shown in Figure \ref{fig:3capsconnections}.  In order to construct $K_{min}$, the $\alpha$ and $\beta$ arcs must twist $p$ times requiring $2p - 2$ bends.  Since there are only $2p - 3$ available in a minimal diagram such a diagram of $K_{min}$ cannot be minimal.

For case $2$ refer to the arc connecting $A_O$ to $B_I$ by $\alpha$, the arc connecting $C_O$ to $A_I$ by $\beta$ and the arc connecting $B_O$ to $C_I$ by $\gamma$.  Reasoning as before, we find that the $\beta$ and $\gamma$ arcs must twist.  In addition, the endpoint labeled $A_I$ must be above the endpoint labeled $B_O$ and the endpoint labeled $B_I$ must be above the endpoint labeled $A_O$.  Furthermore, one may construct $K_{min}$ so that the $\beta$ and $\gamma$ arcs complete $p - 2$ half-twists as shown in Figure \ref{fig:3capsExample}.  The upward arc connecting the two $O$ markings requires at least one cusp, while the twisting of $\beta$ and $\gamma$ requires $2p - 4$ bends, thus using all available upward oriented bends.  Since $p \geq 5$ the twisting of $\beta$ and $\gamma$ requires at least three half-twists, and hence, contains a Type 1 configuration.  A similar argument to that given for the configurations shown in Figure \ref{fig:2capscomparison2} will show that indeed the $X$ extrema must be nested as shown in in the top left diagram of Figure \ref{fig:3capsconnections2}.  The construction described above requires that the downward oriented cusp be placed between the relative maximum and minimum labeled with $O$-markings, leading to a twist as shown on the righthand side of Figure \ref{fig:3capsExample}.  Indeed, there are other possibilities for how this third downward arc (connecting a relative maximum and minimum labeled with $X$-markings via a single cusp) is placed in the diagram relative to the other maxima and minima.  However, if it is not placed as shown in Figure \ref{fig:3capsExample}, it will not produce a minimal diagram for $T_{(p,2)}$.

\begin{center}
\begin{figure}[h]
\centering
\includegraphics[scale=.7]{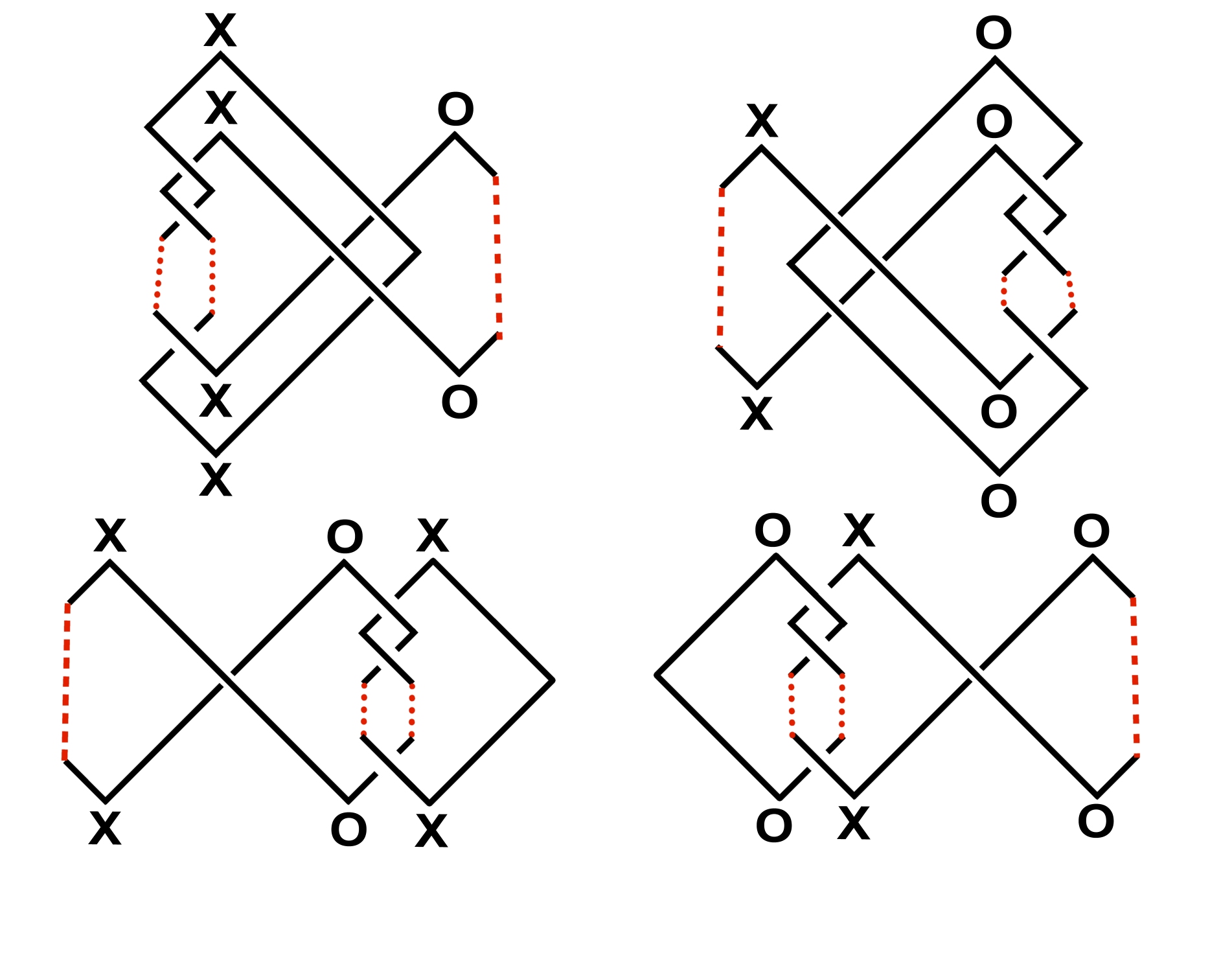}
\caption{$\omega(G) = -p - 1$.}
\label{fig:3capsconnections2}
\end{figure}
\end{center}

\begin{center}
\begin{figure}[h]
\centering
\includegraphics[scale=.4]{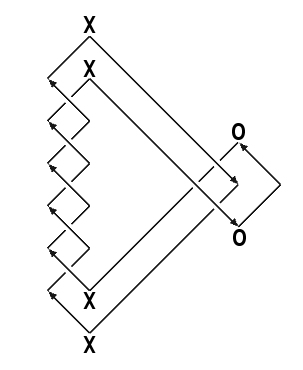}
\caption{$\omega(G) = -p - 1$ for $p = 7$.}
\label{fig:3capsExample}
\end{figure}
\end{center}

\begin{center}
\begin{figure}[h]
\centering
\includegraphics[scale=.8]{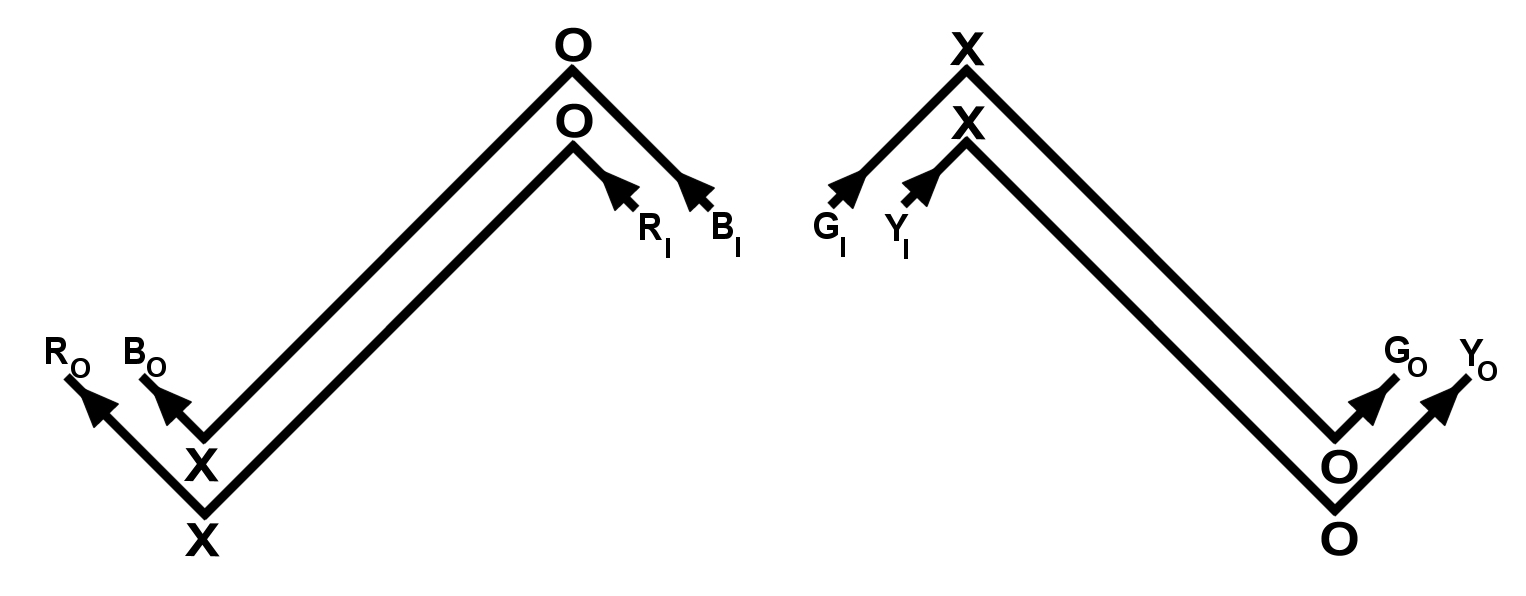}
\caption{$\omega(G) = -p - 2$.}
\label{fig:4capsbase}
\end{figure}
\end{center}

\bigskip
\noindent {\bf Case 4:  $\omega(G) = -p - 2$.}

In this case $D_L = 0$, $U_L = 2p - 4$, and $D_R = U_R = 4$.  For such a Legendrian front projection, there will be $4$ relative maxima and $4$ relative minima.  Since there are no downward oriented cusps, each relative maximum must connect to a relative minimum by a single edge.  Therefore, for each relative maximum marked with an $X$ there must be a corresponding relative minimum marked with an $O$.  Since there are $4$ bends each of types $D_R$ and $U_R$ there must be two relative maxima marked with an $X$ and two marked with an $O$ (the same is true for relative minima).  Since the outgoing edges (those with subscript $O$ in Figure \ref{fig:4capsbase}) must connect to incoming edges (those with subscript $I$ in Figure \ref{fig:4capsbase}) we find $9$ cases for how the free ends may be connected by upward oriented arcs.  Denote a connection between two endpoints with a colon.

\begin{center}
\begin{enumerate}
 \item $G_O : B_I$, $Y_O : R_I$, $B_O : G_I$, $R_O : Y_I$
 \item $G_O : B_I$, $Y_O : R_I$, $R_O : G_I$, $B_O : Y_I$
 
 \item $G_O : B_I$, $B_O : R_I$, $Y_O : G_I$, $R_O : Y_I$
 \item $Y_O : B_I$, $G_O : R_I$, $B_O : G_I$, $R_O : Y_I$
 \item $Y_O : B_I$, $G_O : R_I$, $R_O : G_I$, $B_O : Y_I$
 \item $Y_O : B_I$, $B_O : R_I$, $R_O : G_I$, $G_O : Y_I$
 
 \item $R_O : B_I$, $G_O : R_I$, $Y_O : G_I$, $B_O : Y_I$
 \item $R_O : B_I$, $Y_O : R_I$, $B_O : G_I$, $G_O : Y_I$
 \item $R_O : B_I$, $B_O : R_I$, $Y_O : G_I$, $G_O : Y_I$ 
\end{enumerate}
\end{center}

Of these $9$ cases only cases $2$, $3$, $4$, $6$, $7$, and $8$ represent knots.  The remaining cases have more than one component.  Cases $3$, $6$, $7$, and $8$ are all handled in the same way.  We will show the result in Case $3$.  Each arc has one end directed upward.  By choosing one of the arcs in Figure \ref{fig:4capsbase} and following the upward end, we connect it with one of the other three arcs in Figure \ref{fig:4capsbase}.  Then, the other pair of arcs in Figure \ref{fig:4capsbase} must be connected by an upward arc.  Therefore to construct $K_{min}$ in this case we must choose one of the configurations shown in Figure \ref{fig:4capsGB} and pair it with one of the configurations shown in Figure \ref{fig:4capsRY}.  We outline the arguments for each pairing below, and summarize the results in Table 1.

\begin{center}
\begin{figure}[h]
\centering
\includegraphics[scale=.7]{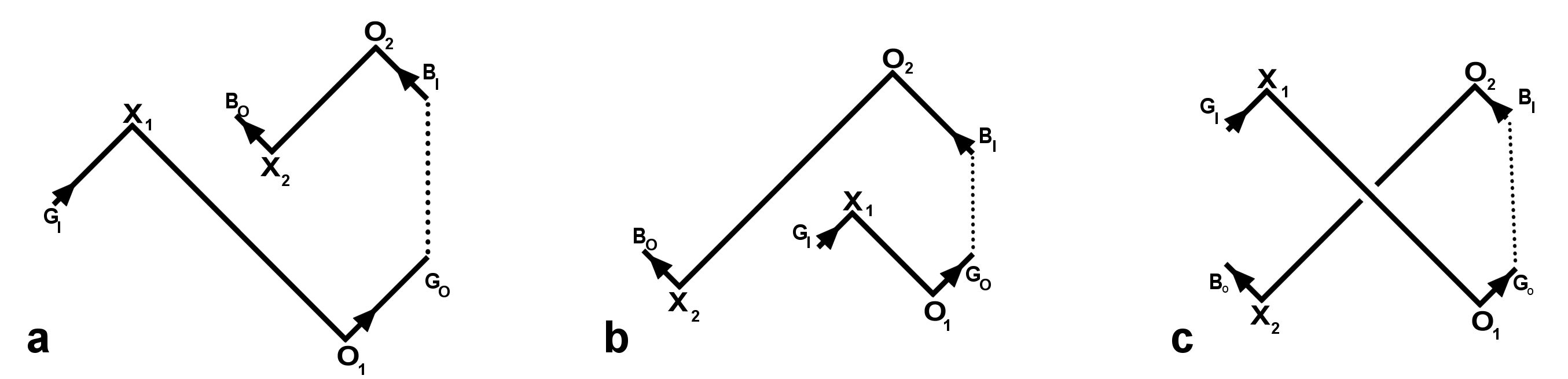}
\caption{Possible configurations of the $G$ and $B$ arcs for Cases $2$ and $3$.}
\label{fig:4capsGB}
\end{figure}
\end{center}

\begin{center}
\begin{figure}[h]
\centering
\includegraphics[scale=.7]{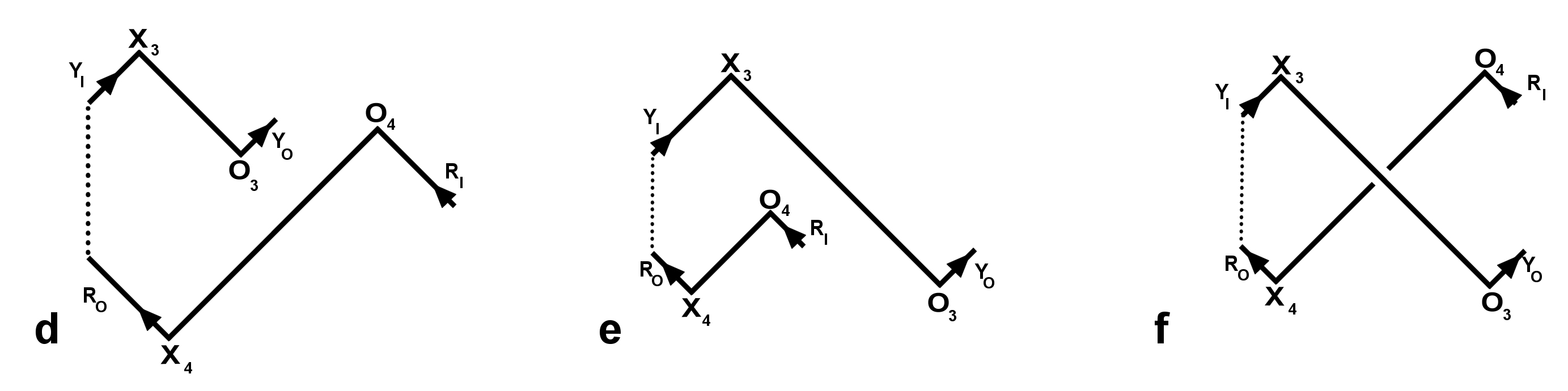}
\caption{Possible configurations of the $R$ and $Y$ arcs for Case $3$.}
\label{fig:4capsRY}
\end{figure}
\end{center}

\medskip

\noindent {\bf ad:} Since $B_O$ connects to $R_I$ via upward cusps $O_4$ must lie above $X_2$.  Either the segment connecting $X_1$ to $O_1$, denoted $XO_1$, crosses the segment connecting $X_4$ to $O_4$, denoted $XO_4$, or not. If $XO_1$ and $XO_4$ do not cross (as in Figure \ref{fig:4capsad}) then $X_1$ cannot lie above $O_3$, but this is required to connect $Y_O$ to $G_I$.  If $XO_1$ and $XO_4$ do cross (as in Figure \ref{fig:4capsad2}) then we cannot construct $T_{(p,2)}$.

\medskip

\begin{center}
\begin{figure}[h]
\centering
\includegraphics[scale=.7]{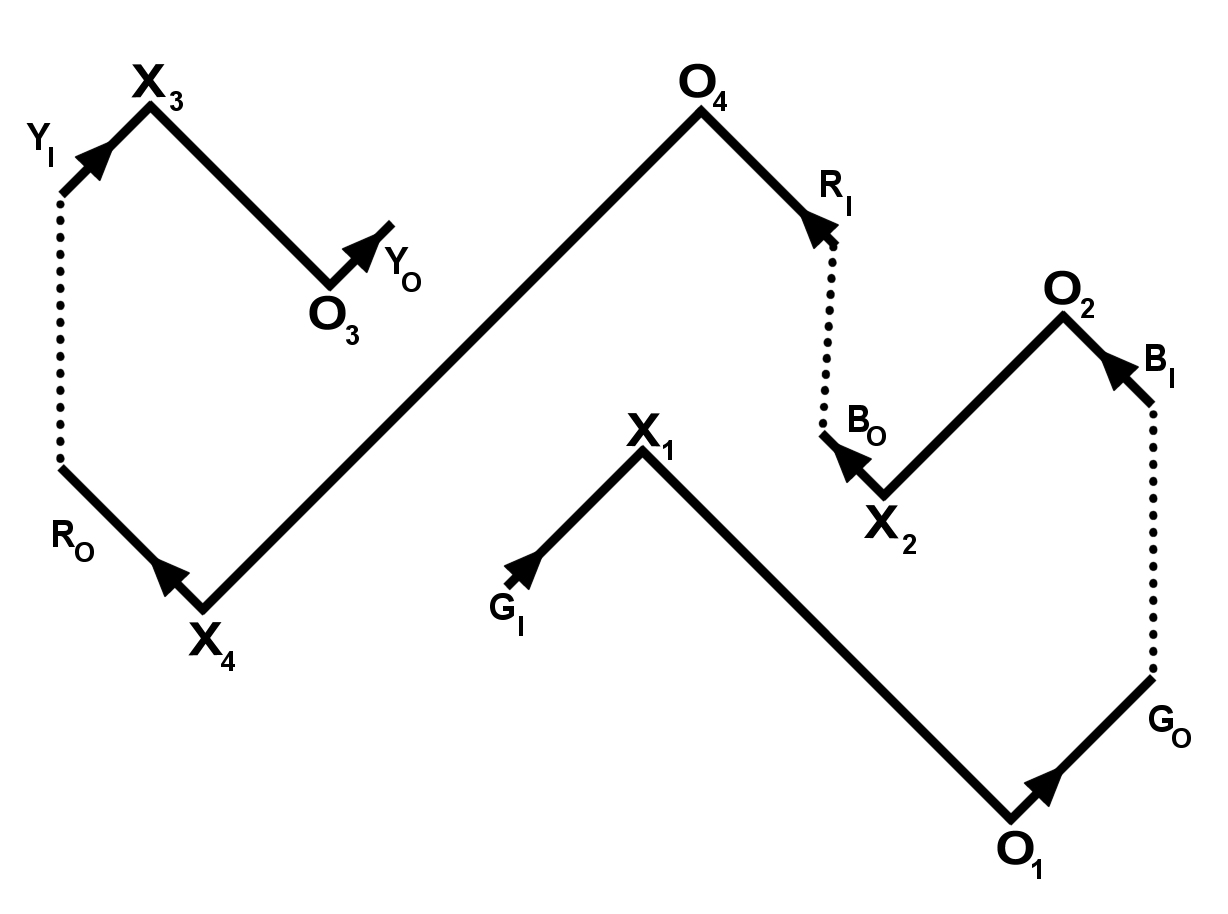}
\caption{Case $ad$ where $XO_1$ and $XO_4$ do not cross.}
\label{fig:4capsad}
\end{figure}
\end{center}

\begin{center}
\begin{figure}[h]
\centering
\includegraphics[scale=.7]{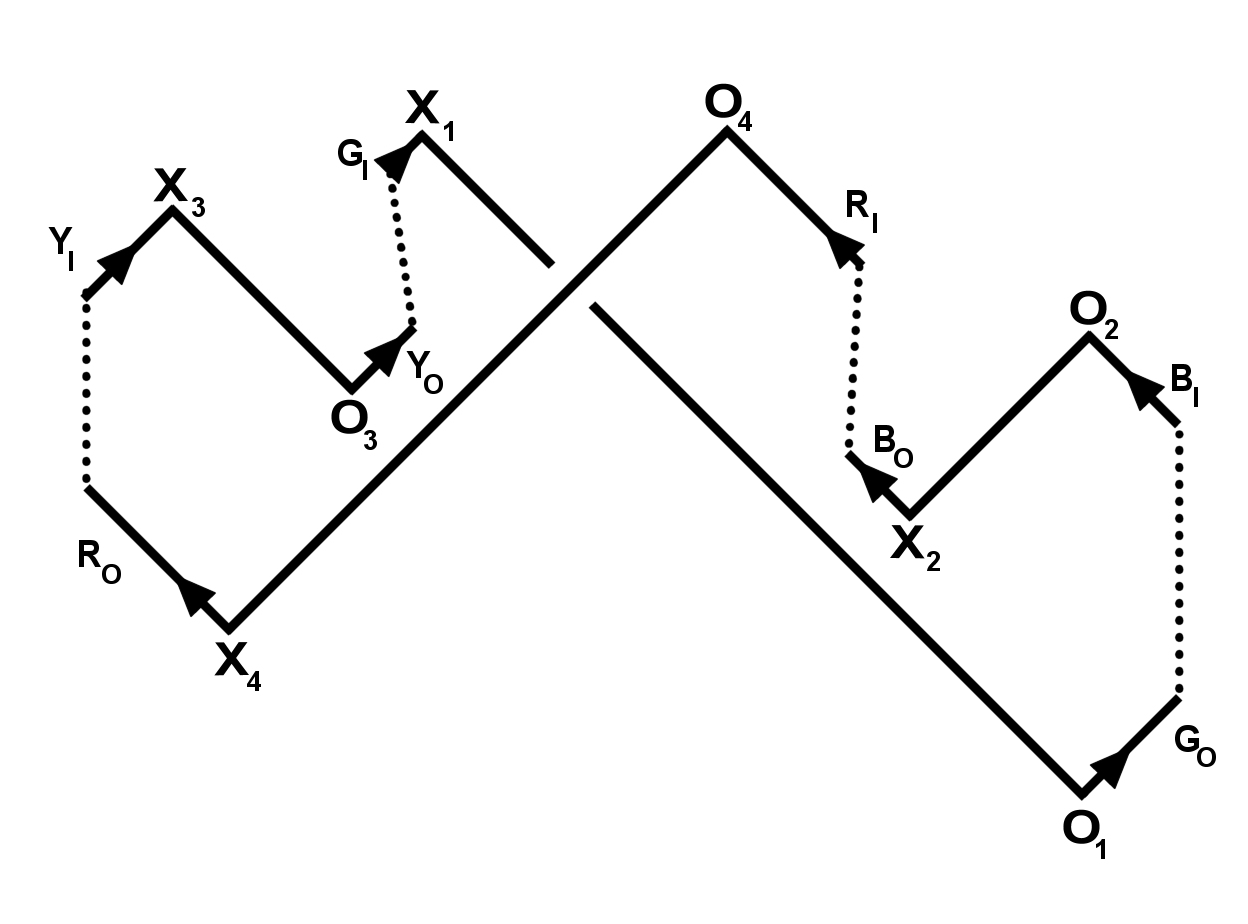}
\caption{Case $ad$ where $XO_1$ and $XO_4$ do cross.}
\label{fig:4capsad2}
\end{figure}
\end{center}

\noindent {\bf ae and bd:}  Since in either case $Y_O$ connects to $G_I$ via upward cusps $O_3$ must lie below $X_1$.  This ensures that $X_2$ is not below $O_4$, hence $R_I$ cannot connect to $B_O$.

\medskip

\noindent {\bf af:}  Since $B_O$ connects to $R_I$ via upward cusps, $O_4$ must lie above $X_2$.  A similar argument shows that $O_3$ must lie below $X_1$.  The resulting diagram cannot be arranged so as to represent a minimal $T_{(p,2)}$.  

\medskip

\noindent {\bf be:}  Since $Y_O$ connects to $G_I$ via upward cusps $X_21$ must lie above $O_3$.  Since $O_3 \preceq X_1 \preceq O_2$ and $O_4 \preceq X_3$ in order for $R_I$ to connect with $B_O$ via upward cusps it must be that $X_2 \preceq O_4$ and hence segments $XO_3$ and $XO_1$ must cross, and we cannot construct $T_{(p,2)}$.  

\medskip

\noindent {\bf bf, cd and ce:}  Since $B_O$ connects to $R_I$ via upward cusps $O_4$ must lie above $X_2$.  Also, since $Y_O$ connects to $G_I$ via upward cusps, $X_1$ must lie above $O_3$.  The resulting diagram cannot be arranged so as to represent a minimal $T_{(p,2)}$.  

\medskip

\noindent {\bf cf:}  Since $B_O$ connects to $R_I$ via upward cusps $O_4$ must lie above $X_2$.  Similarly, since $Y_O$ connects to $G_I$ via upward cusp, $X_1$ must lie above $O_3$.  It is possible to construct $T_{(p,2)}$ as shown in Figure \ref{fig:4capscase2}, but it cannot be minimal.  

\begin{center}
\begin{figure}[h]
\centering
\includegraphics[scale=.7]{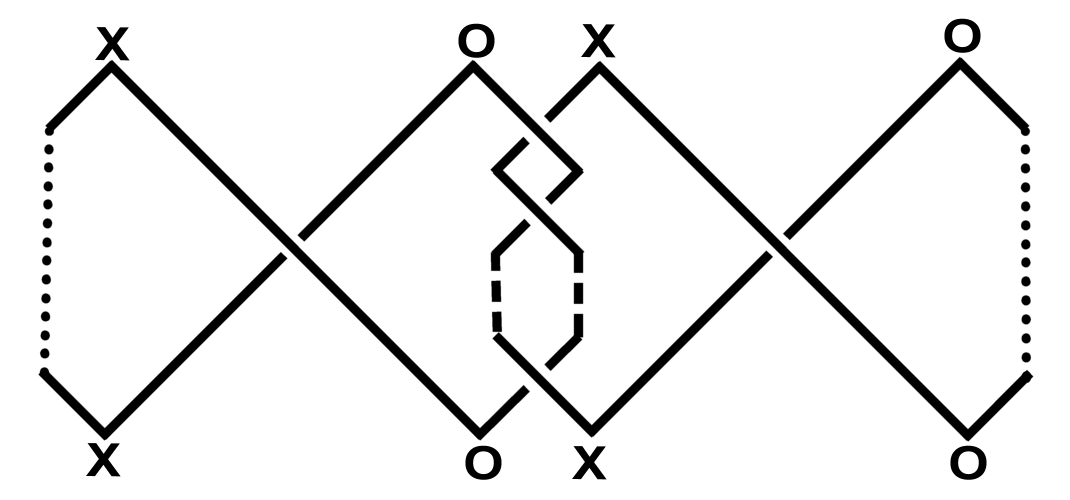}
\caption{$\omega(G) = -p - 2$.}
\label{fig:4capscase2}
\end{figure}
\end{center}

\medskip

\begin{center}
\begin{table}[h]
\label{caseList}
\caption{}
\begin{tabular}{|c|l|}
\hline
Case & Reason \\
\hline
$ad$ & Cannot produce the right knot type. \\
\hline
$ae$ & Configurations cannot be closed up to produce a knot. \\
\hline
$af$ & Cannot produce the right knot type. \\
\hline
$bd$ & Configurations cannot be closed up to produce a knot. \\
\hline
$be$ & Cannot produce the right knot type.\\
\hline
$bf$ & Cannot produce the right knot type. \\
\hline
$cd$ & Cannot produce the right knot type. \\
\hline
$ce$ & Cannot produce the right knot type. \\
\hline
$cf$ & Produces a non-minimal diagram. \\
\hline
\end{tabular}
\end{table}
\end{center}

\medskip

For Case $2$ (Case $4$ is similar) we choose one of the three configurations shown in Figure \ref{fig:4capsGB} and one of the three shown in Figure \ref{fig:4capsRY2}.  Because the configurations in Figures \ref{fig:4capsGB} and \ref{fig:4capsRY2} are the same up to labelling we can reduce the number of cases considered (e.g. Case $ah$ is the same as $bg$).  Table $2$ summarizes the results.

\begin{center}
\begin{figure}[h]
\centering
\includegraphics[scale=.7]{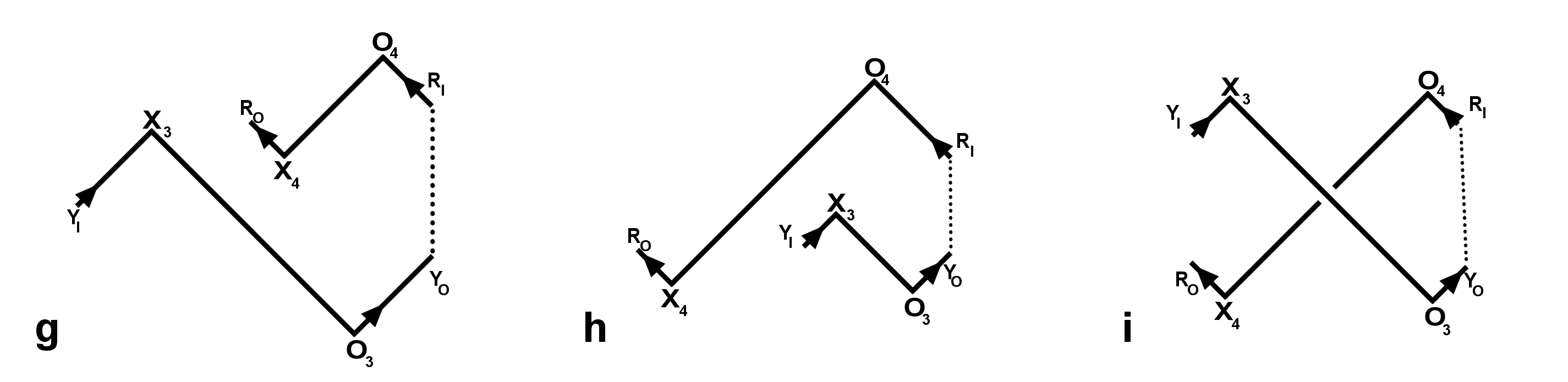}
\caption{Possible configurations of the $R$ and $Y$ arcs for Case $2$.}
\label{fig:4capsRY2}
\end{figure}
\end{center}

\medskip

\noindent {\bf ag:}  Since $R_O$ connects to $G_I$ via upward cusps $X_1$ must lie above $X_4$.  However, this means that there is no way for $X_3$ to be placed above $X_2$, which is necessary in order for $B_O$ to connect with $Y_I$ via upward cusps.  

\medskip

\noindent {\bf ah:}  Since $B_O$ connects $Y_I$ via upward cusps $X_3$ must lie above $X_2$.  Similarly, since $R_O$ connects to $G_I$ via upward cusps $X_1$ must lie above $X_4$.  These two conditions force segments $XO_4$ and $XO_1$ to cross.  In this configuration, it is not possible to produce a minimal diagram for $T_{(p,2)}$.  

\medskip

\noindent {\bf ai and bi:}  As above, $X_1$ must lie above $X_4$ and $X_3$ must lie above $X_2$.  While it is possible to form a $(p,2)$ torus knot from this configuration, it will not be minimal since all twisting must occur on the right-most upward arcs.

\medskip

\noindent {\bf bh:}  Since $R_O$ connects to $G_I$ via upward cusps, $X_1$ must lie above $X_4$.  This means that $X_3$ cannot lie above $X_2$.  However, $X_3$ must lie above $X_2$ if $B_O$ is to connect to $Y_I$ via upward cusps. 

\medskip

\noindent {\bf ci:}  An example of this case is shown in Figure \ref{fig:4capscase1}.  Note that following a similar argument as was given in the 2 maxima/minima case (c.f. Figure \ref{fig:2capscomparison2}) we find that the maxima and minima in this case must be nested as shown in Figure \ref{fig:4capscase1}.  To form $K_{min}$ it is necessary for the pair of arcs on the left to twist $i$ times, and the pair of arcs on the right to twist $j$ times, where exactly one of $i,j$ is odd and the other is even, lest there be two components.  Furthermore, $i + j$ must be at least $p - 2$.  The $i$ half-twists will require $2i$ bends, while the $j$ half-twists will require $2j$ bends.  Thus the total number of bends required will be $2i + 2j = 2(p - 2) = 2p - 4$.  Since at least one of $i,j$ is greater than $1$ (because $p \geq 5$), at least one of the pairs of twisted arcs must introduce a Type $1$ configuration.  

\medskip

Table $2$ summarizes the above results.  Note that completing each construction will lead to the wrong knot type or a non-minimal diagram, unless we choose case $ci$, in which case, at least one Type $1$ configuration will be present.

\begin{center}
\begin{table}[h]
\caption{}
\label{caseList2}
\begin{tabular}{|c|l|}
\hline
Case & Reason \\
\hline
$ag$ & Configurations cannot be closed up to produce a knot. \\
\hline
$ah$ & Cannot produce the right knot type. \\
\hline
$ai$ & Produces a non-minimal diagram.  \\
\hline
$bh$ & Configurations cannot be closed up to produce a knot. \\
\hline
$bi$ & Cannot produce the right knot type. \\
\hline
$ci$ & See above. \\
\hline
\end{tabular}
\end{table}
\end{center}

\begin{center}
\begin{figure}[h]
\centering
\includegraphics[scale=.7]{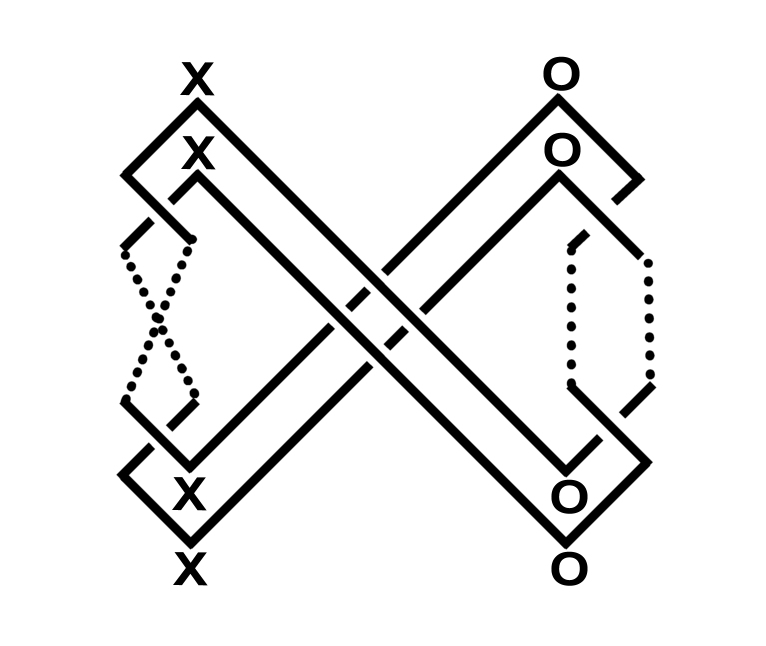}
\caption{$\omega(G) = -p - 2$.}
\label{fig:4capscase1}
\end{figure}
\end{center}

Finally, since all grid diagrams for $K_{min}$ fail to lift due to the appearance of Type 1 configurations, $c_\ell(K_{min}) > p + 2$.

\end{proof}

In the preceding theorem it was required that $p \geq 5$.  For $p = 3$ (the trefoil) Legendrian cube number does not distinguish between $K_{min}$ and $K_{max}$.  The above proof fails for the $p = 3$ case because introducing a single half-twist for the pair of arcs in Case $3$ ($3$ maxima and $3$ minima) is sufficient to build the trefoil, thus avoiding the introduction of Type 1 configurations.

\end{document}